\documentclass[11pt,reqno,tbtags]{amsart}
\usepackage[utf8]{inputenc}
\usepackage[english]{babel}

\usepackage{mathrsfs}
\usepackage{amsmath}
\usepackage{amsfonts}
\usepackage{amsthm}
\usepackage{latexsym}
\usepackage{amsthm}
\usepackage[dvips]{graphics,color}
\usepackage{graphicx} 
\usepackage{psfrag}
\usepackage{subfigure}

\newcommand{\led}{\langle}
\newcommand{\red}{\rangle}

\newcommand{\Ups}{\Upsilon}

\newcommand{\rate}{\sigma}
\newcommand{\speed}{\mathscr V}
\newcommand{\LL}{\mathscr L}
\newcommand{\hJ}{\hat J}
\newcommand{\hY}{\hat Y}
\newcommand{\diff}{\Delta}
\newcommand{\dd}{\,\text d}
\newcommand{\tb}{\tilde \beta}
\newcommand{\ta}{a^*}
\newcommand{\sa}{\hat a}
\newcommand{\tr}{r^*}
\newcommand{\sr}{\hat r}
\newcommand{\tR}{R^*}
\newcommand{\sR}{\hat R}

\newcommand{\rf}[1]{\overline{#1}}
\newcommand{\ff}[1]{\underline{#1}}
\newcommand{\ett}{\mathscr A}
\newcommand{\bigo}{\mathcal O}
\newcommand{\tva}{\mathscr B}

\newtheorem{claim}{Claim}
\newtheorem{rem}{Remark}

\title[First-passage percolation on ladder-like graphs]
{First-passage percolation on ladder-like graphs with 
heterogeneous exponential times.}

\date{\today{}}

\author{Henrik Renlund}
\address{Department of Mathematics, Uppsala University, PO Box 480,
S-751 06 Uppsala, Sweden}
\email{henrik.renlund@math.uu.se}
\urladdr{http://www2.math.uu.se/$\sim$renlund/}

\begin{document}

\begin{abstract}
We determine the asymptotic speed of the first-passage percolation process on some ladder-like graphs 
(or width-2 stretches) when the times associated with different edges are independent and
exponentially distributed but not necessarily all with the same mean. The method uses a
particular Markov chain associated with the first-passage percolation process and 
properties of its stationary distribution. 
\end{abstract}

\maketitle

\section{Introduction}
Consider a graph $G$ with vertex set $V$ and (undirected) edge set $E\subset V\times V$. If there
is an edge $e=\led v,v' \red$ joining vertices $v$ and $v'$, they are said to be adjacent. 
A path $\pi(v,v')$ between $v$ and $v'$ is an alternating sequence of vertexes 
and edges $(v_0,e_1,v_1,\ldots,e_n,v_n)$ such that $e_i=\led v_{i-1},v_i \red$ for $i=1,\ldots,n$, 
$v_0=v$ and $v_n=v'$.

We think of each edge $e$ as being associated with a (typically non-negative) random
time $\xi_e$, formally we define
 $\xi=\{\xi_e,e\in E\}$ to be a stochastic process indexed by the edges of the graph.
We define the time $T\pi(v,v')$ of
a path $\pi(v,v')$ to be 
\[ T\pi(v,v')=\sum_{e\in \pi(v,v')} \xi_e\quad\mbox{and}\quad
 T(v,v')=\inf_{\pi(v,v')} T\pi(v,v') \]
to be the shortest time of any path between $v$ and $v'$. $T(v,v')$
is called the first passage time from $v$ to $v'$ and
is the subject of investigation in first passage percolation,
see e.g.\ \cite{SW78}. 

We may also consider first passage times between sets of vertices. If $V_1$ and
$V_2$ are such sets, then 
\[ T(V_1,V_2)=\inf_{v\in V_1, v'\in V_2} T(v,v') \]
is the first passage time from $V_1$  to $V_2$. 

Usually, all first passage times under consideration are from some fixed and given
set $V_0$ to a varying set $V'$, in which case we denote by
$T(V')$ the first passage time from $V_0$ to $V'$, or 
$T(v')$ if $V'$ is the singleton set $\{v'\}$, and refer to this simple as
the first passage time of $V'$.  

We can think of this as a model for a contagious disease. At time
zero a subset $V_0$ of the vertices are infected and subsequently the disease
spreads to adjacent vertices. The time for the infection to be submitted from
$v$ to $v'$ along the edge $\langle v,v'\rangle$ is $\xi_{\led v,v' \red}$.

A typical question, and indeed the one we shall be concerned with, is how \emph{fast}
does the infection spread? This of course requires some metric on the set of vertices. We 
shall look at cases of \emph{ladder-like} graphs with ``height'' as the measure of 
distance and independent edge times, each having  an exponential distribution. 

\vspace{0.2cm}
\noindent\textbf{Acknowledgement.} The author wishes to thank Robert Parviainen for
encouragement and help that borders on collaboration, as well as professor Svante Janson for
interesting discussions on the mathematical technicalities of this manuscript.

\section{First passage percolation on ladder-like graphs} \label{section2}
A ladder-like graph, or width-2 stretch, is a graph whose vertex set is 
$V=\mathbb N\times\{0,1\}$, $\mathbb N$ denoting $\{0,1,2,\ldots\}$, 
with the ordinary Euclidean metric, considered when each vertex is placed in
the (real) plane $\mathbb R^2$.
Our interest is on a particular class $\LL$ of such graphs where vertices \emph{may} be adjacent if
they are no more than Euclidean distance $\sqrt 2$ away from each other and, in addition,
the graph is translational invariant, see Figure \ref{LL}. By translational invariant we mean
that we can superimpose the graph on itself by shifting it one unit length to the right in Figure \ref{LL}.

\begin{figure}[hbt]
\begin{center}
\psfrag{x}{0}
\psfrag{y}{1} 
\psfrag{z}{2}
\psfrag{w}{3}
\includegraphics[scale=0.5]{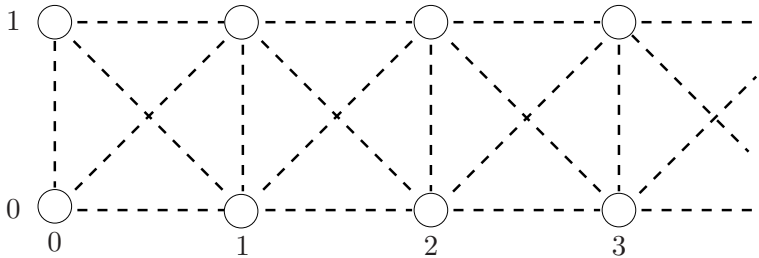}
\caption{Part of the vertices of a ladder-like graph. 
The possible edges of a graph in $\LL$ are dashed. Note that it must also
be translational invariant.}
\label{LL}
\end{center}
\end{figure}

A vertex $(x,y)$ in such a graph is said to be at height $x\in \mathbb N$
and level $y\in\{0,1\}$. We will consider $V_0=\{(0,0),(0,1)\}$ to be infected at 
time zero, and we  consider first passage times of the height $n$, i.e.\ the set 
$V_n=\{(n,0),(n,1)\}$.

It should be noted that the Euclidean metric only serves as to define the class
of graphs we are interested in. Once this is done, we only consider ``height'' 
as a measure of length in the graph. In this sense, the ``length'' of the (possible)
diagonal edges are not longer that that of the (possible) horizontal edges.

Typically the focus is on the a.s.\ limit 
\[ \rate=\rate(G,\xi)=\lim_{n\to \infty}\frac{T(V_n)}{n}, \]
called the time constant. Now, the first passage times on these
graphs are subadditive and the almost sure convergence of $T(V_n)/n$ 
relies on the subadditiv ergodic theorem. See Proposition 4.1 of \cite{Sch09}
for the details. In our calculations it is rather the asymptotic speed $\speed$ of percolation
that appears naturally, but this is just the inverse of the time constant, $\speed=1/\rate$.

The investigation of the rate, or speed, of percolation on ladder-like graphs began with 
\cite{FGS06}, which consider \emph{the ladder}, i.e. where all vertices at distance
1 are connected by an edge.
 That paper gives a method of calculating $\rate$ when $\xi_e$:s are independent
with the same discrete distribution, as well as a method for getting arbitrarily 
good bounds for the same quantity when the distribution is continuous (and well behaved).

A few years later \cite{Sch09} and \cite{Ren10} independently studied the case of 
having independent and identical exponential times. In \cite{Sch09}, six graphs in
$\LL$ are considered and explicit expressions for $\rate$ are found for three of them, using
recursive distributional equations. It should be noted that these 
six graphs are in fact ``all'' graphs in the class $\LL$ as far as $\rate$
is concerned, in the sense that
any other graph is either trivial 
or it has the same $\rate$ as one of these six.  

Another method for finding $\rate$ is employed in \cite{Ren10}, which studies the ladder. 
There a Markov chain is found and the speed of percolation can be calculated 
``at stationarity'', a method that will be repeated below in a more
general context. The method also
allows the study of another aspect of the percolation process called
the \emph{residual times}, defined in \cite{Ren10}.

Now, we fix some notation that will be used throughout this paper.
Let
\begin{equation}\label{N} N_t=\sup_{i\in\{0,1\}}\{x: T[(x,i)]\leq t\} \end{equation}
denote -- in the interpretation of the model of a contagious disease -- the height of the infection
at time $t$.

Let 
\begin{equation}\label{M} M_t=\sup\{ x: \max\{T[(x,0)], T[(x,1)]\}\leq t \} \end{equation}
denote the largest height at which both levels have an infected vertex. 

Now  define the \emph{front process} $F_t$ as 
\begin{equation}\label{F} F_t=N_t-M_t. \end{equation} 
In some situations $F_t$ is a Markov chain on $\mathbb N$. As we shall see, the problem
of determining the percolation speed $\speed$ is more or less equivalent to finding
the stationary distribution of this Markov chain.  

We will denote by $\varrho_n \sim \vartheta_n$ the property that
 \[ \displaystyle \lim_{n\to\infty}\frac{\varrho_n}{\vartheta_n}=1 \]
for real sequences $\varrho_n$ and $\vartheta_n$.
Let $z^{\ff n}$ and $z^{\rf n}$ denote falling and rising factorials, respectively, 
which may be expressed as
\begin{align*}
 z^{\ff n} &=z(z-1)\cdots(z-n+1)=\frac{\Gamma(z+1)}{\Gamma(z+1-n)},\;\mbox{and} \\ 
z^{\rf n} &=z(z+1)\cdots (z+n-1)=\frac{\Gamma(z+n)}{\Gamma(z)}.
\end{align*}

\section{The ladder with heterogeneous exponential times}
In this section we generalize the result of \cite{Ren10} to a ladder with
heterogeneous exponential times. Let each horizontal edge be associated with an
exponential time with intensity $\lambda_{\mathrm h}>0$  (i.e.\ an exponential with expectation
$1/\lambda_{\mathrm h}$) and each vertical edge be associated with an 
exponential time with intensity $\lambda_{\mathrm v}>0$. Random variables associated to
different edges are independent. Notice, since we may equally
well measure intensity in units of $\lambda_{\mathrm h}$, there is no loss of generality in saying 
that $\lambda_{\mathrm h}=1$. Let us do so and denote $\lambda_{\mathrm v}$ simply by $\lambda$. 

Note that \cite{Ren10} is concerned with the case $\lambda=1$. 

\begin{figure}[hbt]
\begin{center}
\psfrag{x}{0}
\psfrag{y}{1} 
\psfrag{z}{2}
\psfrag{w}{3}
\psfrag{a}{4}
\psfrag{b}{5}
\psfrag{c}{6}
\psfrag{d}{7}
\includegraphics[scale=0.3]{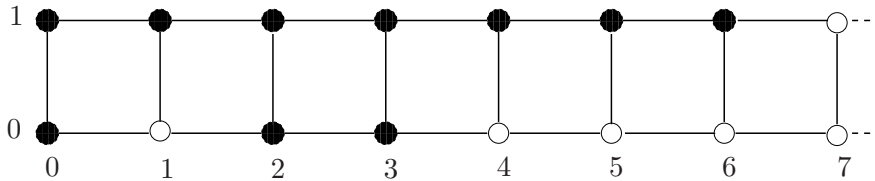}
\caption{Infected nodes at time $t$ marked as black. Here $M_t=3, N_t=6$ and $F_t=3$. 
Not shown in the picture is that each vertical edge has intensity $\lambda$ and each horizontal edge
has intensity $1$.}
\label{front_ex}
\end{center}
\end{figure}

Now, $F_t$, defined by (\ref{N}), (\ref{M}) and (\ref{F}), is a continuous time Markov chain on
$\mathbb N$. This fact becomes clear as we write down the intensity matrix $Q$ of the process.
Denote by $t'=\inf\{t>t_0: F_t\neq F_{t_0}\}$ the first time, after $t_0$, that the $F$-process 
 changes its value.

Consider first the case of $F_{t_0}=0$, which means that $M_{t_0}$ and $N_{t_0}$ both equal, say, $n$. 
Then both $(n,0)$ and $(n,1)$ are infected, and there is no infected node at any greater height. 
The only states that can be infected at $t'$
are $(n+1,0)$ or $(n+1,1)$, both resulting in $F_{t'}=1$, and both as a result of the infection spreading
along a horizontal edge. Informally, we will say that these edges \emph{lead to} the state $F_{t'}=1$, 
which is rather inaccurate since edges connect different vertices and not different states of the front process,
but a convenient terminology. So, thus far we have established that  the intensity from state 0 to state 1 is 2 
and the intensity away from state 0 is -2.  

If $F_{t_0}=3$, as in Figure \ref{front_ex},  there are 3 vertical edges leading to  
states 2, 1 and 0 respectively. There are 2 horizontal edges leading to states 2 and 4,
respectively. Thus, the intensities to 0 and 1 are both $\lambda$, to 2 it is $1+\lambda$ and to 4 the intensity is 1.

In a similar manner, one derives the following form of the intensity matrix:
\begin{equation} \label{Q}
Q=\left( \begin{array}{rrrrrrr}
-2 & 2 & 0 & 0 & 0  \\
 1+\lambda & -2-\lambda & 1 & 0 & 0 \\
 \lambda & 1+\lambda & -2-2\lambda & 1 &0  &\ldots \\
 \lambda & \lambda & 1+\lambda & -2-3\lambda & 1 \\
 \lambda & \lambda & \lambda & 1+\lambda & -2-4\lambda \\
 \lambda & \lambda & \lambda & \lambda & 1+\lambda &   \\
  &  &\vdots  &&  &  \ddots
\end{array} \right). \end{equation}
This Markov chain is irreducible, so if we can find a stationary distribution 
$\Pi=(\pi_0,\pi_1,\ldots)$, on $\mathbb N$, such that $\Pi Q=0$, then this is the unique stationary distribution
of $F_t$, and as $t\to\infty$, $F_t$ will converge to it. 

Now, our main question is: what is the asymptotic speed of percolation?
  When $F_{t_0}=0$ the intensity by which $N_t$ increases
is 2, since there are  
two horizontal edges from the
state $F_{t_0}=0$ to the state $N_{t'}=N_{t_0}+1$ (and no other edge
 makes any difference to either of these processes). 
When $F_{t_0}>0$, there is one horizontal
edge leading to $N_{t'}=N_{t_0}+1$ (and no other edge that makes any difference
to the $N_t$-process). 
Hence, at stationarity we get the speed of percolation as
\begin{equation}\label{rop} \speed=2\pi_0+1(1-\pi_0)=1+\pi_0. \end{equation}

Next, we proceed to find $\Pi$ through the equation $\Pi Q=0$. One way of doing this is
to express each $\pi_n$, $n>0$, in terms of $\pi_0$; $\pi_n=a_n\pi_0-b_n$.

Now, the first entry of $0=\Pi Q$ 
is
\begin{equation*}
 0=-2\pi_0+\pi_1+\lambda\sum_{j\geq 1} \pi_j.
\end{equation*}
 As $\sum_{j\geq 0}\pi_j=1$, which implies $\sum_{j\geq 1}\pi_j=1-\pi_0$, we get 
\begin{equation}\label{1}
 0=-(2+\lambda)\pi_0+\lambda+\pi_1\quad\mbox{yielding}\quad\pi_1=(2+\lambda)\pi_0-\lambda.
\end{equation}
Repeating this procedure for the equations corresponding to columns 2 and 3 of $Q$, yields, after 
isolating $\pi_2$ and $\pi_3$,
\begin{align}
 \pi_2 &= (2\lambda^2+7\lambda+2)\pi_0-(2\lambda^2+3\lambda),\quad\mbox{and} \label{2} \\
 \pi_3 &= (6\lambda^3+26\lambda^2+22\lambda+2)\pi_0-(6\lambda^3+14\lambda^2+6\lambda). \label{3}
\end{align}
Next, taking the difference of equations corresponding to columns $k$ and $k+1$
for any $k\geq 2$ yields 
\[ 0=\pi_{k-1}-(\lambda k+3)\pi_{k}+(\lambda (k+2)+3)\pi_{k+1}-\pi_{k+2}, \]
which we shall rewrite as
\begin{equation}\label{n}
 \pi_n= (\lambda n+3)\pi_{n-1}-(\lambda (n-2)+3)\pi_{n-2}+\pi_{n-3}, \; n\geq 4.
\end{equation}
We formally define $a_n$ and $b_n$ by the relation $\pi_n=a_n\pi_0-b_n$, for
$n\geq 1$. Thus, for $n=1,2,3$ they are defined by equations (\ref{1}), (\ref{2})
and (\ref{3}), respectively, and for $n\geq4$ can be found iteratively through
(\ref{n}), with $a_n$ or $b_n$ in place of $\pi_n$. 

As necessarily $\pi_n\to 0$ as $n\to\infty$, we can find $\pi_0$ as
\[ \pi_0=\lim_{n\to\infty} \frac{b_n}{a_n}. \]

This far we have copied the procedure of \cite{Ren10}, which proceeds to 
prove the link to Bessel functions, along the lines of how this fact was
originally discovered. Here, we start instead with the Bessel functions.

\subsection{The Bessel functions}
The Bessel functions of the first and second kind, $J_\nu(z)$ and
$Y_\nu(z)$, as described in e.g.\ \cite{Wat22},
both satisfy the recursion
\begin{equation}\label{Brec}
 C_{\nu+1}(z)+C_{\nu-1}(z)=\frac{2\nu}{z}\cdot C_\nu(z).
\end{equation}
Define, for any (real numbers) $n$, $A$ and $B\neq 0$, 
\[
\hJ_n=\hJ_n(A,B)=J_{n+A/B}(2/B) \quad\mbox{and}\quad \hY_n=\hY_n(A,B)=Y_{n+A/B}(2/B).
\]
Then $\hJ_n$ and $\hY_n$ both satisfy 
\begin{equation}\label{Hrec}
 \hat C_{n+1}+\hat C_{n-1}=(A+Bn)\hat C_n.
\end{equation}
Define the function
\begin{equation}\label{Upsdef}
 \Ups_n=\Ups(n,m,A,B)=\pi[\hJ_{n}\hY_m-\hJ_m\hY_{n}].
\end{equation}
This function inherits recursion (\ref{Hrec}) (in
parameter $n$) 
\begin{equation} \label{Ups}
\Ups_{n+1}+\Ups_{n-1}=(A+Bn)\Ups_n. 
\end{equation}
Next, define 
\[ \diff_n=\diff(n,m)=\diff(n,m,A,B)=\Ups_n-\Ups_{n-1}. \]

When $m$ and/or $A$ and $B$ are clear from the context, or unimportant, we write 
$\Delta_n$ or $\Delta(n,m)$ instead of the lengthy $\Delta(n,m,A,B)$. 
As the next claim makes clear, this function is designed with (\ref{n}) in mind.
\begin{claim}
$\Delta_n$  satisfies 
\begin{equation}\label{delta}
 \diff_n=[1+A+B(n-1)]\diff_{n-1}-[1+A+B(n-3)]\diff_{n-2}+\diff_{n-3}.
\end{equation}
\end{claim}

\begin{proof}
By using the definition $\diff_k=\Ups_k-\Ups_{k-1}$ on the right hand side
of (\ref{delta}), rearranging terms so that relation (\ref{Ups}) can be applied (as indicated
in (\ref{u-b})),  we see that the right hand side of (\ref{delta}) equals
\begin{align}
&\underbrace{[A+B(n-1)]\Ups_{n-1}}_{\Ups_{n}+\Ups_{n-2}}
\underbrace{-2[A+B(n-2)]\Ups_{n-2}}_{-2\Ups_{n-1}-2\Ups_{n-3}}
  \underbrace{+[A+B(n-3)]\Ups_{n-3}}_{+\Ups_{n-2}+\Ups_{n-4}} + \label{u-b} \\
&\quad\quad\quad  +\Ups_{n-1}-2\Ups_{n-2}+2\Ups_{n-3}-\Ups_{n-4}  \nonumber \\
&=\Ups_{n}-\Ups_{n-1},
\end{align}
which is the same as the left-hand side of (\ref{delta}).
\end{proof}
\begin{claim}
 \begin{align} 
\Ups(n,n,A,B) &=0\quad\mbox{and} \\
\Ups(n+1,n,A,B) &=B.
 \end{align}
\end{claim}

\begin{proof}
The first assertion follows trivially from the definition (\ref{Upsdef}).

For the second assertion, we need two important relations for the Bessel functions. The first is that
\begin{equation}\label{fir1}
 zJ_\nu'(z)=\nu J_\nu(z)-zJ_{\nu+1}(z),\quad \forall \nu.
\end{equation}
and the same holds with $Y_\nu$ replacing $J_\nu$, see
 3.2(4), p.\ 45 and 3.56(4), p.\ 66 of \cite{Wat22}, resp.
The second is a fact relating to the Wronskian, namely that
\begin{equation}\label{sec2}
 \frac2{\pi z}=J_\nu(z)Y'_\nu(z)-J_\nu'(z)Y_\nu(z), \quad \forall \nu, z\neq 0,
\end{equation}
see 3.63(1), p.\ 76 of \cite{Wat22}.

(\ref{fir1}) inserted in (\ref{sec2}) gives
\begin{equation*}
 \frac 2z=\pi[J_{\nu+1}(z)Y_\nu(z)-J_\nu(z)Y_{\nu+1}(z)]
\end{equation*}
and $z=2/B$ gives the desired result.
\end{proof}

\begin{rem}
 It should be noted that the $\Ups$-function does not solve a recursion
of the form (\ref{delta}) with \emph{arbitrary} initial conditions. Indeed,
by relation (\ref{Ups}) we get 
\[ \Delta_{n+1}-\Delta_n=(A+Bn-2)\Ups(n,m,A,B). \]
 Hence, with $A$ and $B$ considered fixed, two 
points of the sequence of $\Delta$:s (such that $A+Bn-2\neq 0$) is enough
 to determine $m$ - which may not be unique, but in general
not arbitrary - which in turn determines the entire sequence. 
\end{rem}

From the preceding claim, and relation (\ref{Ups}), it is straightforward to calculate:
\begin{claim} \label{diffval}
\begin{align*}
\diff(m,m,A,B) &=B, \\
 \diff(m+1,m,A,B) &=B, \quad\mbox{and}\\
\diff(m+2,m,A,B) &=[A+B(m+1)-1]B.
\end{align*}
\end{claim}

Next, we need some asymptotic properties of $\Ups_n$.  

\begin{claim} \label{c:upsasympt}As $n\to\infty$,
 \begin{equation}\label{upsasympt}
  \Ups(n,m,A,B)\sim \hat J_m\Gamma(n+A/B)B^{n+A/B}.
 \end{equation}
\end{claim}

\begin{proof}
 \cite{H64} lists the following asymptotic relations, for fixed $x$ and
$n$ tending to infinity,
\[ J_n(x) \sim \frac{1}{\sqrt{2\pi n}}\left(\frac{ex}{2n}\right)^n
 \quad\mbox{and}\quad Y_n(x)\sim -\sqrt\frac{2}{\pi n}\left(\frac{2n}{ex}\right)^n, \]
hence, by the well-known Stirling formula 
$\Gamma(n+1)\sim \sqrt{2\pi}n^{n+1/2}e^{-n}$ we conclude
\begin{align}
 \hat J_n &\sim \frac1{\Gamma(n+A/B+1)B^{n+A/B}}  \label{jmot0} \\ 
 \hat Y_n &\sim -\frac 1\pi\Gamma(n+A/B)B^{n+A/B} \nonumber
\end{align}
from which the claim follows.
\end{proof}

Now, we have the tools to continue with the problem of finding the speed of the percolation process.

\subsection{Finding the stationary distribution}
Now, (\ref{delta}) fits (\ref{n}) if we set $A=2+\lambda $ and $B=\lambda $. One way of trying to
describe $a_n$ and $b_n$ is through some linear combination of $\Delta$:s. 
By inspection of Claim \ref{diffval}, we  choose to work with
$\Delta(n,1)$ and $\Delta(n,2)$. Then one finds that
\begin{align*}
 a_n &= \frac{2\lambda ^2+8\lambda +5}{\lambda }\cdot\Delta(n,1)-\frac{\lambda +3}{\lambda }\cdot\Delta(n,2), \;\mbox{and} \\
 b_n &= \frac{2\lambda ^2+4\lambda +1}{\lambda }\cdot\Delta(n,1)-\frac{\lambda +1}{\lambda }\cdot\Delta(n,2),
\end{align*}
is true for $n=1,2,3$ which makes it true for every $n$. 

Next, with the aid of Claim \ref{c:upsasympt}, we calculate
\begin{equation} \label{pi0}\pi_0=\lim_{n\to\infty}\frac{b_n}{a_n}= 
\frac
{(2\lambda ^2+4\lambda +1)J_{2+2/\lambda }(2/\lambda )-(\lambda +1)J_{3+2/\lambda }(2/\lambda )}
{(2\lambda ^2+8\lambda +5)J_{2+2/\lambda }(2/\lambda )-(\lambda +3)J_{3+2/\lambda }(2/\lambda )}, \end{equation}
and thereafter, a long, but straightforward, calculation\footnote{The analogous calculation to
that of (2.19) in \cite{Ren10}.}
 yields, for  $n\geq 1$,
\[ \pi_n=c(\hat J_{n-1}-\hat J_n), \quad\mbox{with}\;c=\frac{2}{(2\lambda ^2+8\lambda +5)\hat J_1-(\lambda +3)\hat J_2}, \]
after which it can be verified that $\Pi$ is indeed a distribution since,
\[ \sum_{k=0}^n \pi_k=1-c\hat J_n\to 1, \]
as $\hat J_n\to 0$ by (\ref{jmot0}).

The speed $\speed$ of percolation in this model as given by (\ref{rop}) is $1+\pi_0$, 
with $\pi_0$ as in (\ref{pi0}), also depicted in Figure \ref{speed} with $\lambda $ ranging from 
close to zero\footnote{It appears to be difficult to evaluate the speed for very small values
of $\lambda$, as $B=1/\lambda$ becomes very large. In Figure \ref{speed} it looks as if
the speed at $\lambda=0$ is $\approx 1.2$. However, this is not the case. } to 20. 

\begin{figure}[hbt]
\includegraphics[scale=0.4]{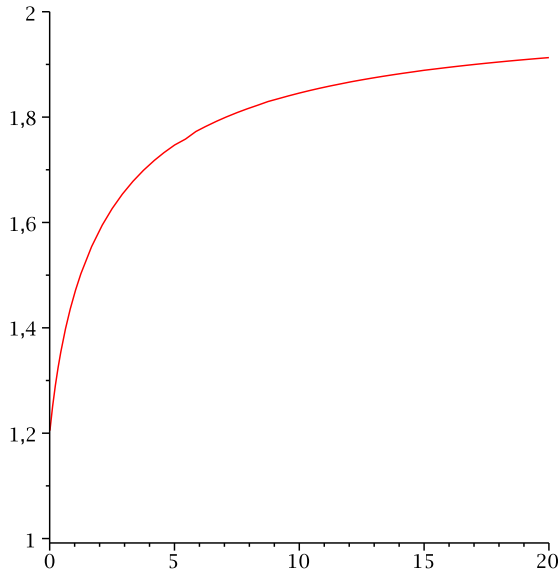}
\caption{The speed of percolation on the ladder as a function of (exponential) intensity $\lambda $ on the vertical edges 
(as measured in units of the intensity of the horizontal edges), plotted by Maple. Note
that the speed at $\lambda=0$ is 1.}
\label{speed}
\end{figure}

Although not defined, the speed as a function of $\lambda =0$ is 1, since $\lambda =0$ must be thought of as
not having any vertical edges at all (or equivalent that the times
associated with these are infinite) and as such the time to reach $V_n$ is
the time it takes until the first of two independent Poisson processes reaches 
$n$. Although this is always a bit faster than the time for a single such process,
the effect wears off as $n$ tends to infinity. 

Similarly, the speed at $\lambda =\infty$ is 2, if we interpret $\lambda =\infty$ as having zero
time associated with vertical edges. Then infection between adjacent vertices
on different levels is immediate, and the time it takes to move one step
up the ladder is the minimum of two exponential variables each having one unit of
intensity, i.e.\ an exponential random time having intensity 2. 

\begin{figure}[hbt]
 \subfigure[$1+\frac{7J_5(2)-2J_4(2)}{15J_5(2)-4J_4(2)}$]{\includegraphics[scale=0.40]{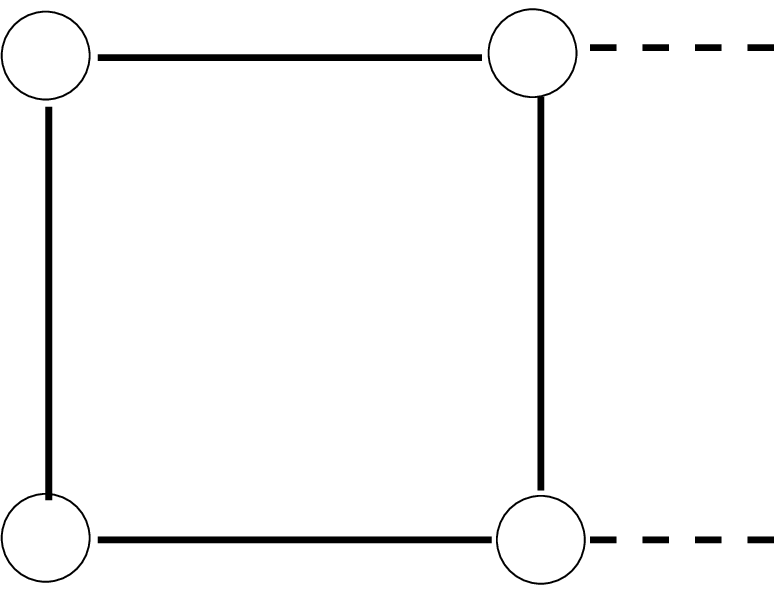}}
 \subfigure[$1+\frac{17J_4(1)-3J_3(1)}{29J_4(1)-5J_3(2)}$]{\includegraphics[scale=0.40]{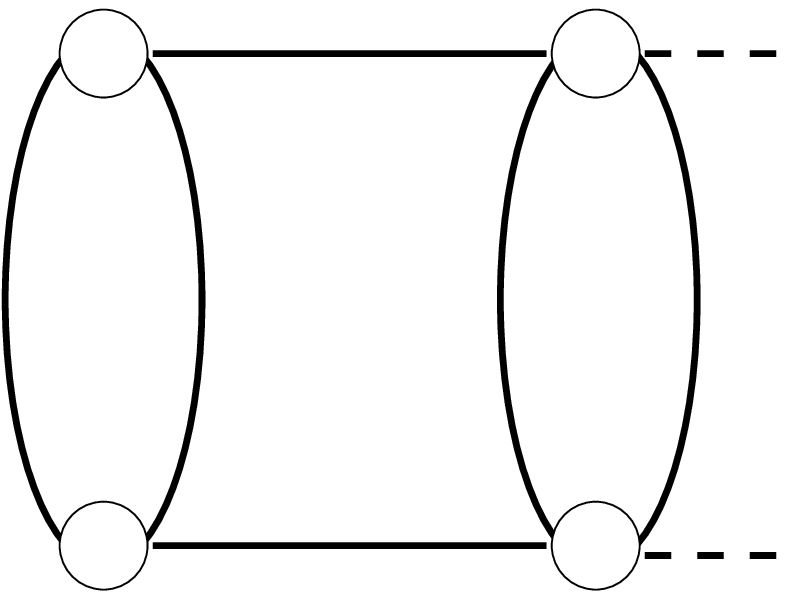}}
 \subfigure[$2+2\cdot\frac{7J_7(4)-3J_6(4)}{15J_7(4)-4J_6(4)}$]{\includegraphics[scale=0.40]{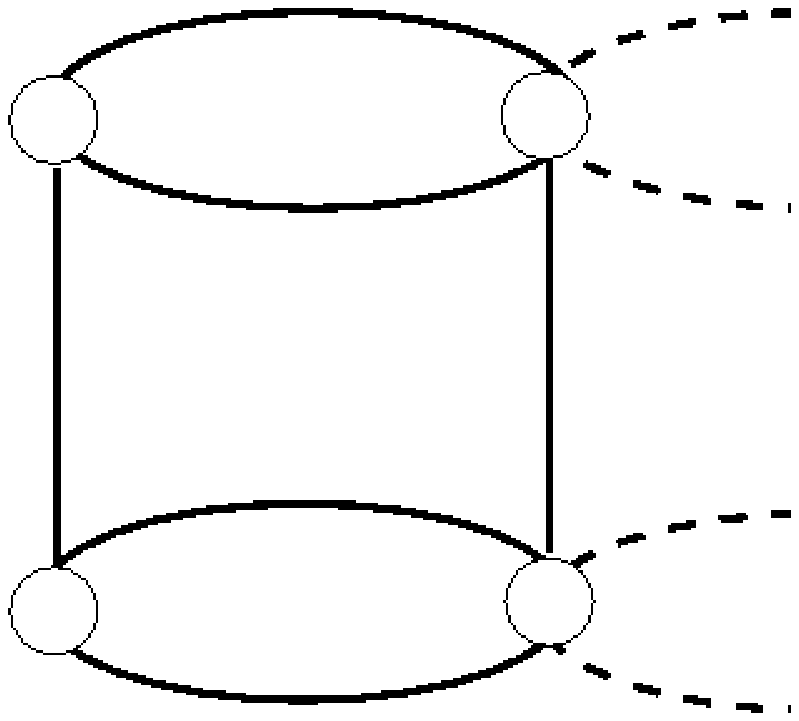}}
\caption{The speed of percolation when all edges are associated with random variables having intensity 1. Numerical
values to two digits are, from left to right, 1.47, 1.59 and 2.74. }
\label{exempel}
\end{figure}

Having the intensity  2 associated with an edge is equivalent to having two (independent) edges associated with
intensity 1, since the minimum time of two intensity 1 exponential random variables is an intensity 2 exponential
random variable. Some examples of the speed of percolation on different ladders are shown in 
Figure \ref{exempel}. Notice that in the rightmost graph, we calculate the speed as two times that
of having $\lambda =1/2$.

\section{Adding diagonals to the ladder}\label{diag}
Next, we see what happens when diagonals are added to the ladder. For this model we cannot find 
a solution through Bessel function, as was the case with the ladder. Neither can we consider arbitrary 
intensities associated to the diagonals, these must be of the same intensity as the horizontal, 
else the front process becomes intractable. 

\begin{figure}[hbt]
 \psfrag{a}{$v$}
 \psfrag{b}{$v'$}
\includegraphics[scale=0.5]{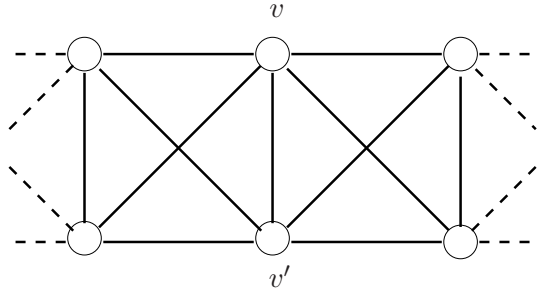}
\caption{Intensities associated with vertical edges have intensity $\lambda $, as measured in the unit 
of intensity that is associated with both the horizontal and diagonal edges. Note that the graph
is the same if we interchange $v$ and $v'$ (without breaking any edges).} 
\label{ekvivalens}
\end{figure}

Part of the graph we are now considering is depicted in Figure \ref{ekvivalens}. In this figure
we see two nodes $v$ and $v'$. If we interchange the positions of these two 
nodes, without breaking any edges, the graph is effectively unaltered. The vertical edges are still
the same (they are now ``upside down'', but this is irrelevant as edges are undirected). 
Four horizontal edges have become diagonal and vice versa, but as these are all associated with 
the same intensities, nothing has essentially changed. This is, of course, the reason we need
to have the same intensities on horizontal and diagonal edges. 

Consider the front process, defined by (\ref{N}),(\ref{M}) and (\ref{F}), at time $t$. If we disregard
all nodes below height $M_t$, there are $2^k$ different sets of infected nodes (at $t$) that could
yield $\{F_t=k\}$. From any such state the intensity to move to to any other state that yields
$\{F_{t'}=k'\}$ is the same, and as such the front process will be a Markov chain on $\mathbb N$. 

In Figure \ref{Fis2} we illustrate this when the front process is found to be 2 at $t$. 

\begin{figure}[hbt]
 \subfigure{\includegraphics[scale=0.3]{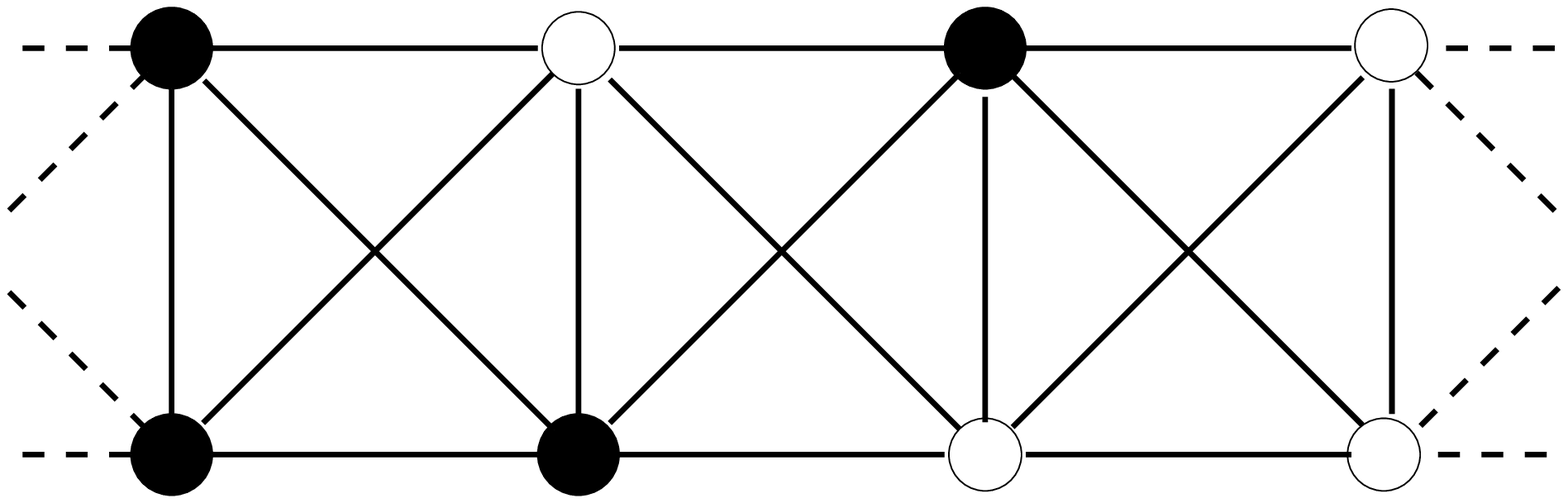}}\quad\quad
 \subfigure{\includegraphics[scale=0.3]{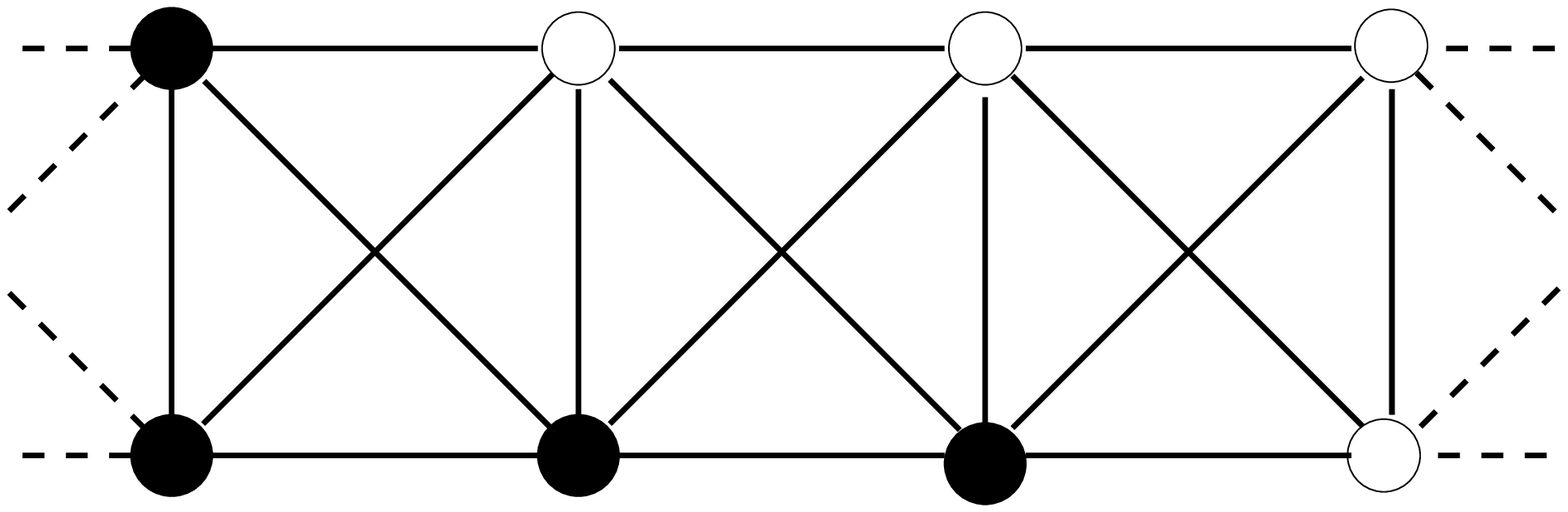}}
\caption{Two of the four possible sets of infected nodes at $t$ (disregarding nodes below
height $M_t$)  that would yield $F_t=2$. The infected nodes are marked as black. 
The remaining two sets are found by mirroring the cases above through a horizontal line.}
\label{Fis2}
\end{figure}

Now, it is straightforward to compute the intensity matrix $Q$ of the front process, e.g.\ when
in state 2, as in Figure \ref{Fis2}, then intensities to state 0 is $1+\lambda $, to state 1 is $3+\lambda $ and
to state 3 is $2$, and by similar considerations we get

\begin{equation} \label{Q2}
Q=\left( \begin{array}{rrrrrrr}
-4 & 4 & 0 & 0 & 0  \\
 2+\lambda  & -4-\lambda  & 2 & 0 & 0 \\
 1+\lambda  & 3+\lambda  & -6-2\lambda  & 2 &0  &\ldots \\
 1+\lambda  & 2+\lambda  & 3+\lambda  & -8-3\lambda  & 2 \\
 1+\lambda  & 2+\lambda  & 2+\lambda  & 3+\lambda  & -10-4\lambda  \\
 1+\lambda  & 2+\lambda  & 2+\lambda  & 2+\lambda  & 3+\lambda  &   \\
  &  &\vdots  &&  &  \ddots
\end{array} \right). \end{equation}

We proceed to find the stationary distribution $\Pi=(\pi_0,\pi_1,\ldots)$.
Analogously to the computation of (\ref{1}), (\ref{2}), (\ref{3}) and (\ref{n}) we get 
\begin{equation}\label{n2}
 \pi_n=[(2+\lambda )n+3]\pi_{n-1}-[(2+\lambda )(n-2)+4]\pi_{n-2}+2\pi_{n-3},\quad n\geq 4,
\end{equation}
and if we let $\pi_n=c_n\pi_0-d_n$, for $n\geq 1$, we have
\begin{align*}
(c_1,c_2,c_3) &= (5+\lambda ,28+17\lambda +2\lambda ^2,226+226\lambda +68\lambda ^2+6\lambda ^3),\;\mbox{and} \\
(d_1,d_2,d_3) &= (1+\lambda ,8+9\lambda +2\lambda ^2,66+98\lambda +44\lambda ^2+6\lambda ^3),
\end{align*}
and $c_n,d_n$ for $n\geq 4$ can be determined via relation (\ref{n2}) with
$c_n$ or $d_n$ in place of $\pi_n$. As in the previous section, we aim to determine
$\pi_0$ via
\begin{equation}\label{limfrac} \lim_{n\to\infty} \frac{d_n}{c_n}. \end{equation}
Once, this is done, the speed of percolation $\speed$ is 
\[ \speed=4\pi_0+2(1-\pi_0)=2(1+\pi_0). \]

\subsection{Turning to generating functions}

Next, we employ the recursive techniques of \cite{Jan10} to get an explicit expression
for $c_n$ and $d_n$. For ease of comparison, we will adopt the notation of that
paper. \cite{Jan10} considers recursively defined sequences $a_n$ of the form
\begin{equation} \label{srec}
a_n=\sum_{i=1}^K[\alpha_i(n-i)+\tb_i]a_{n-i},\quad n\geq K, 
\end{equation}
where  $\alpha_1,\ldots,\alpha_K$ and $\tb_1,\ldots,\tb_K$ are given numbers, and
the initial part of the sequence $a_0,\ldots,a_{K-1}$ is fixed. Also, the numbers have been 
normalized as to give $\alpha_1=1$. We may work in this setting if we, for $n\geq 0$ set
\[ \sa_n=c_{n+1}(2+\lambda )^{-n} \quad \mbox{and}\quad \ta_n=d_{n+1}(2+\lambda )^{-n}.  \]
Our main interest is $\lim_n d_n/c_n$ which will equal $\lim_n \ta_n/\sa_n$.

So, we aim to find a formula for $a_n$ given by (\ref{srec}) where $K=3$ and
$\alpha_1=1$, $\alpha_3=0$ and $\alpha_2$ typically is negative. Set $\alpha=-\alpha_2$. 
Then the parameters 
\begin{align*} 
\alpha &=\frac1{2+\lambda }, \\
\tb_1 &= 2+\frac3{2+\lambda }=\frac{7+2\lambda }{2+\lambda }, \\
\tb_2 &= -\frac{1}{2+\lambda }-\frac4{(2+\lambda )^2}=-\frac{6+\lambda }{(2+\lambda )^2}\;\mbox{and}\\
\tb_3 &= \frac2{(2+\lambda )^3},
\end{align*}
 are the
same for the sequences $\sa_n$ and $\ta_n$, which then ``only'' differ in
initial values
\begin{align*}
(\sa_1,\sa_2,\sa_3) &= (5+\lambda ,\frac{28+17\lambda +2\lambda ^2}{2+v},\frac{226+226\lambda +68\lambda ^2+6\lambda ^3}{(2+\lambda )^2}),\;\mbox{and} \\
(\ta_1,\ta_2,\ta_3) &= (1+\lambda ,\frac{8+9\lambda +2\lambda ^2}{2+v},\frac{66+98\lambda +44\lambda ^2+6\lambda ^3}{(2+\lambda )^2}).
\end{align*}
Following \cite{Jan10} we define the generating function 
\[ A(z)=\sum_{k=0}^\infty a_kz^k \]
for the sequence $a_n$. Next, relation (\ref{srec}), with $K=3$ and parameters as above, yields
\begin{equation}\label{difeq} A(z)q(z)=A'(z)z^2p(z)+r(z), \end{equation}
with
\begin{align}
 p(z) &=\sum_{i=1}^3\alpha_iz^{i-1}=1-\alpha z, \nonumber\\ 
 q(z) &= 1-\sum_{i=1}^3\tilde\beta_iz^i=1-\tilde\beta_1z-\tilde\beta_2z^2-\tilde\beta_3z^3,\quad\mbox{and} \nonumber\\
 r(z) &= \sum_{i=0}^2 R_iz^i,\quad\mbox{with} \label{ing} \\
 R_0 &= a_0,\quad R_1=a_1-\tb_1 a_0,\quad\mbox{and}\quad R_2=a_2-(1+\tb_1)a_1-\tb_2 a_2. \nonumber
\end{align}
Let $I\subset \mathbb R$ be such that $0\in I$ and  $p(z)\neq 0$ on $I$. Theorem 3.2 (i) and (iv) of \cite{Jan10}
reveals that there exists a solution $A$ to (\ref{difeq}) in $I$ and any such solution satisfies the asymptotic expansion
$A(z)=\sum_{n=0}^N a_nz^n+\bigo(z^{N+1})$, as $z\to 0$.

Now, of the three polynomials $p$, $q$ and $r$ it is only the latter that differs between $\sa_n$ and $\ta_n$, so let us
denote these specific polynomials $\sr(z)=\sum \sR_iz^{i-1}$ and $\tr(z)=\sum\tR_iz^{i-1}$, respectively. 
The coefficients simplify to
\begin{align}
 \sR_0 &=5+\lambda , \nonumber \\
 \sR_1 &=\frac{28+17\lambda +2\lambda ^2}{2+\lambda }-\frac{7+2\lambda }{2+\lambda }(5+\lambda )=-\frac7{2+\lambda }=-7\alpha, 
\nonumber \\
 \sR_2 &=\frac{226+226\lambda +68\lambda ^2+6\lambda ^3}{(2+\lambda )^2}-\left(1+\frac{226+226\lambda +68\lambda ^2+6\lambda ^3}{(2+\lambda )^2}\right) 
\nonumber\\
&\phantom{=}+\frac{6+\lambda }{(2+\lambda )^2}(5+\lambda ) = \frac 4{(2+\lambda )^2}=4\alpha^2, 
\nonumber\\
 \tR_0 &=1+\lambda ,
\nonumber\\
 \tR_1 &=\frac{8+9\lambda +2\lambda ^2}{2+v}-\frac{7+2\lambda }{2+\lambda }(1+\lambda )=\frac 1{2+\lambda }=\alpha, \quad\mbox{and} 
\nonumber\\
 \tR_2 &=\frac{66+98\lambda +44\lambda ^2+6\lambda ^3}{(2+\lambda )^2}-\left(1+\frac{66+98\lambda +44\lambda ^2+6\lambda ^3}{(2+\lambda )^2}\right) 
\nonumber\\
&\phantom{=}+\frac{6+\lambda }{(2+\lambda )^2}(1+\lambda ) = 0. \label{allaR}
\end{align}
Still using the notation of \cite{Jan10}, we implicitly define $\gamma$ and the rational
function $g(z)$ via 
\begin{equation*}
 \frac{q(z)}{z^2p(z)}=z^{-2}+\gamma z^{-1}+g(z),
\end{equation*}
 so that 
 \begin{equation*}
 \gamma=\alpha-\tilde\beta_1=-\frac{6+2\lambda}{2+\lambda} \quad\mbox{and}\quad 
g(z)=\frac{\alpha(\alpha-\tb_1)-\tb_2-\tb_3z}{1-\alpha z}.
 \end{equation*}
We fix an antiderivative $G$ to $g$  as  
\begin{equation*}
 G(z)=Cz+D\ln(1-\alpha z), 
\end{equation*}
where 
\begin{equation*}
C=\frac{\tb_3}{\alpha}=2\alpha^2 \quad\mbox{and}\quad 
D=\frac{\tb_3}{\alpha^2}+\frac{\tb_2}{\alpha}+\tb_1-\alpha=\alpha(2+\lambda)=1.
\end{equation*} 
Now, we can write down what in \cite{Jan10} is called \emph{the principal solution}
to (\ref{difeq}), which we consider for $z<0$, 
\begin{align}
 A_0(z) &=
\int_0^\infty (1-zt)^\gamma\exp\left\{-t-G\left(\frac{z}{1-zt}\right)+G(z)\right\}
\frac{r\left(\frac{z}{1-zt}\right)}{p\left(\frac{z}{1-zt}\right)}\dd t 
\nonumber \\
&=
\int_0^\infty \Bigg[ (1-zt)^\gamma \frac{(1-\alpha z)^D(1-zt)^D}{(1-(\alpha+t)z)^D}\exp\left\{-\frac{Ctz^2}{1-zt} \right\}
e^{-t} \nonumber \\
& \quad\quad\quad\quad 
 \cdot\frac{1-zt}{1-(\alpha+t)z}\sum_{i=0}^2 R_i\left(\frac{z}{1-zt}\right)^i\Bigg] \dd t \label{tbc}
\intertext{where we by three Taylor expansions and the relation $(-x)^{\ff n}=(-1)^nx^{\rf n}$  get, as $z\nearrow 0$,}
(\ref{tbc}) &=\sum_{i=0}^2 R_i\int_0^\infty \Bigg[ (1-zt)^{\gamma+D+1-i}
\sum_{j=0}^N\frac {(-\alpha)^jD^{\ff j}z^j}{j!}  \nonumber \\
& \quad    \cdot
\sum_{m=0}^N \frac{(-\alpha-t)^m(-D-1)^{\ff m}z^m}{m!} 
\sum_{k=0}^N\frac {(-Ctz^2)^k}{k!(1-zt)^k}
e^{-t}z^i\Bigg]\dd t 
+\bigo(z^{N+1})\nonumber \\
&= 
\sum_{i=0}^2 R_i
\sum_{j=0}^N\frac {(-\alpha)^jD^{\ff j}}{j!}
\sum_{k=0}^N\frac {(-C)^k}{k!}
\sum_{m=0}^N \frac{(D+1)^{\rf m}}{m!}  
 \nonumber \\
&\quad\quad\quad\quad  \cdot  \underbrace{
\int_0^\infty  (1-zt)^{\gamma+D+1-i-k}t^k(\alpha+t)^me^{-t}\dd t}_{=:I}\cdot z^{i+j+2k+m} . \label{I}
\end{align}
The integral  denoted $I$ 
in (\ref{I}), 
 equals, by Taylor expansion, as $z\nearrow 0$,
\begin{align}
& I=\int_0^\infty  \sum_{l=0}^N \frac{(-t)^l(\gamma+D+1-i-k)^{\ff l}z^l}{l!}t^k(\alpha+t)^me^{-t}\dd t+\bigo(z^{N+1}) 
\nonumber \\
&=
\sum_{l=0}^N \frac{(k+i-\gamma-D-1)^{\rf l}}{l!}z^l\int_0^\infty t^{k+l}
\sum_{n=0}^m\binom mnt^n\alpha^{m-n}e^{-t}\dd t + \bigo(z^{N+1})= 
\nonumber \\ 
\intertext{}
&=
\sum_{l=0}^N \frac{(k+i-\gamma-D-1)^{\rf l}}{l!}z^l\sum_{n=0}^m\binom mn\alpha^{m-n}(k+l+n)!  + \bigo(z^{N+1}). \label{theI}
\end{align}
So, from (\ref{I}) and (\ref{theI}), we get the following formula  
\begin{align}
&A_0 (z) =\sum_{i=0}^2 R_i
\sum_{j=0}^N\frac {(-\alpha)^jD^{\ff j}}{j!}
\sum_{k=0}^N\frac {(-C)^k}{k!}
\sum_{l=0}^N \frac{(k+i-\gamma-D-1)^{\rf l}}{l!} \nonumber \\
& \cdot \sum_{m=0}^N \frac{(D+1)^{\rf m}}{m!}
\sum_{n=0}^m\binom mn\alpha^{m-n}(k+l+n)!
 z^{i+j+2k+l+m}  + \bigo(z^{N+1}), \label{theform}
\end{align}
which simplifies somewhat since $D=1$ in our present application. 
 Using this, and letting
$B_0$ and $B_1$ be the nonzero summands of $\sum_{j=0}^\infty (-\alpha)^jD^{\ff j}/j!$, i.e.\
$B_0=1$ and $B_1=-\alpha$,  and by defining
\begin{align}
 F_i(k,l,m) &= 
\frac{(-C)^k}{k!}\frac{(k+i-\gamma-2)^{\rf l}}{l!}\frac{2^{\rf m}}{m!}\sum_{n=0}^m\binom mn\alpha^{m-n}(k+l+n)!, \label{F_i}
\end{align}
we can write, when considering $A_0(z)$ as an infinite sum, 
\begin{align*}
 A_0(z) &=\sum_{i=0}^2R_i\sum_{j=0}^1 B_j \sum_{k,l,m\geq 0} F_i(k,l,m)z^{i+j+2k+l+m}.
\end{align*}
Further, defining $B_j=0$ if $j\notin\{0,1\}$ and $R_i=0$ if $i\notin\{0,1,2\}$ makes 
\begin{equation}\label{id}
 A_0(z)=\sum_{L=0}^3\;
 \sum_{k,l,m\geq 0}\;
 \sum_{i+j=L}\;B_jR_iF_i(k,l,m) z^{L+2k+l+m}.
\end{equation}
Since $A_0(z)$ can be written on the form (\ref{theform}) this
allows us, by Theorem 3.2 of \cite{Jan10}, to identify the coefficients
of $A_0(z)$ to those of $A(z)$, so from (\ref{id}),
\begin{equation}\label{a_N}
 a_N=\sum_{L=0}^3 \;
\sum_{2k+l+m=N-L} \;
\sum_{i+j=L}\;B_jR_iF_i(k,l,m),
\end{equation}
if we follow the convention that an empty sum is zero. 

Now, to get a handle on $a_N$, consider 
\begin{align}
 &\hat F_i(M) =\sum_{2k+l+m=M} F_i(k,l,m) \nonumber \\
&= \sum_{2k+l+m=M} \frac{(-C)^k}{k!}\frac{(k+i-\gamma-2)^{\rf l} }{l!}(m+1)\sum_{n=0}^m
 \binom{m}{n}\alpha^{m-n}(k+l+n)!. \label{hatF}
\end{align}

\begin{claim}\label{theLi}
\begin{equation} \label{Li}
L_i:=\lim_{M\to\infty}\frac{\hat F_i(M)}{M!M^{\hat\gamma+1+i}} = 
\sum_{k=0}^\infty \frac{(-C)^k}{k!\Gamma(k+1+\hat\gamma+i)}
\int_0^1(1-x)^{\hat\gamma+k+i}xe^{\alpha x}\dd x, 
\end{equation}
where $\hat\gamma=-\gamma-3=-\lambda/(2+\lambda)\in(-1,0]$.
\end{claim}

\begin{proof}
The details are given in Appendix \ref{pLi}.
\end{proof}

We need to find a useful expression
for the integral that appears in (\ref{Li}).
\begin{claim}\label{theIn}
Define, for $n\in\mathbb N$,  
\begin{align}   \label{IoJ}
 I(n) =\int_0^1(1-x)^{\hat\gamma+n}xe^{\alpha x}\dd x \quad\mbox{and} \quad
 J(n) =\int_0^1(1-x)^{\hat\gamma+n}e^{\alpha x}\dd x. 
\end{align}
Then
\begin{align*} 
 I(n) =
\frac{\Gamma(\hat\gamma+n+1)}{\alpha^n}\bigg[ \frac{I(0)-\frac{n}{\alpha}J(0)}{\Gamma(\hat\gamma+1)}+
\frac1{\alpha^2}\sum_{m=1}^{n}\frac{(n+1-m)\alpha^{m}}{\Gamma(\hat\gamma+1+m)}\bigg]. 
\end{align*}
\end{claim}

\begin{proof}
 The details are given in Appendix \ref{pIn}. (Note that the result is true for 
any $\alpha\neq 0$ but that it is essential that $\hat\gamma>-1$.)
\end{proof}

\subsection{Evaluating the limit} Now we return to examination of
the limits $L_i$ given in (\ref{Li}).

By (\ref{Li}), (\ref{IoJ}) and (\ref{Jf}) we have, for $i\in\{0,1,2\}$,
\begin{align}
 L_i &= \sum_{k=0}^\infty\frac{(-C)^k}{k!\Gamma(\hat\gamma+k+1+i)}I(k+i) \nonumber \\
&= \frac{I(0)-\frac i\alpha J(0)}{\alpha^i\Gamma(\hat\gamma+1)}S_3
-\frac{J(0)}{\alpha^{i+1}\Gamma(\hat\gamma+1)} S_4+\frac1{\alpha^{i+2}} S(i), \label{Li2}
\end{align}
where 
\begin{align}
 S_3 &=\sum_{k=0}^\infty \frac{(-C/\alpha)^k}{k!}
=e^{-C/\alpha}=e^{-2\alpha}, \nonumber \\
S_4 &= \sum_{k=0}^\infty\frac{k(-C/\alpha)^k}{k!}=-\frac C\alpha e^{-C/\alpha}=-2\alpha e^{-2\alpha}, 
\quad\mbox{and} \nonumber \\
S(i) &= \sum_{k=1}^\infty\sum_{m=1}^{k+i}f(k,m,i), \quad\mbox{where} \nonumber \\
f(k,m,i)&=\frac{(-1)^k(2\alpha)^k\alpha^m(k+i+1-m)}{k!\Gamma(\hat\gamma+1+m)}. \label{f}
\end{align}
This yields, by rearranging (\ref{Li2}) and using $C/\alpha^2=2$,
\begin{equation}
L_i= \frac1{\alpha^i}
\left( 
\frac{e^{-2\alpha}[I(0)+2J(0)]}{\Gamma(\hat\gamma+1)}-i\cdot\frac{e^{-2\alpha}J(0)}{\alpha\Gamma(\hat\gamma+1)}
+\frac{1}{\alpha^{2}}S(i) 
\right). \label{Lif}
\end{equation}

Recall that $\alpha\in(0,1/2]$ and that $\hat\gamma>-1$ so that $f$ defined in (\ref{f})
satisfies $|f(k,l,m)|<|k+i+1-m|/k!m!$ when $k\geq 0$ and $m\geq 1$, and so
the sum $S(i)$ is absolutely convergent and we are thus free to interchange the
order of summation
\begin{align} 
 S(0) &= \sum_{k=0}^\infty\sum_{m=1}^{k}f(k,m,0)= \sum_{k=1}^\infty\sum_{m=1}^{k}f(k,m,0)=\sum_{m=1}^{\infty}\sum_{k=m}^{\infty}f(k,m,0)
\nonumber \\
&=\sum_{m=1}^{\infty}\sum_{j=0}^{\infty} f(j+m,m,0)=\sum_{j=0}^{\infty}\sum_{m=1}^{\infty}f(j+m,m,0) 
\nonumber \\
&=\sum_{j=0}^\infty (-2\alpha)^{j}(j+1) \sum_{m=1}^\infty \frac{(-1)^m(2\alpha^2)^m}{\Gamma(m+1+\hat\gamma)\Gamma(m+1+j)}=
\mathscr S(0), \label{S0}
\end{align}
if we define
\begin{equation} 
 \mathscr S(k) =\sum_{j=0}^\infty (-2\alpha)^{j}(j+1+k) 
\sum_{m=1}^\infty \frac{(-1)^m(2\alpha^2)^m}{\Gamma(m+1+\hat\gamma)\Gamma(m+1+j)}. \label{scrS}
\end{equation}
Similarly to how we handled $S(0)$ we get
\begin{align}
S(1) &=  \sum_{j=0}^\infty \sum_{m=1}^\infty f(j+m-1,m,1) 
\nonumber\\
&=\sum_{j=0}^\infty (-2\alpha)^{j-1}(j+1) \sum_{m=1}^\infty \frac{(-1)^m(2\alpha^2)^m}{\Gamma(m+1+\hat\gamma)\Gamma(m+1+j-1)} 
\nonumber\\
&=-\frac 1{2\alpha}\sum_{m=1}^\infty\frac{(-1)^m(2\alpha^2)^m}{(m-1)!\Gamma(m+1+\hat\gamma)}+\mathscr S(1) 
\nonumber\\
&= \frac{2^{(1-\hat\gamma)/2}}{2\alpha^{\hat\gamma}}J_{1+\hat\gamma}(2\sqrt 2\alpha)+\mathscr S(1), \label{S1}
\end{align}
where $J$ denotes the Bessel function of the first kind, as well as
\begin{align}
S(2) &=  \sum_{k=0}^\infty\sum_{m=1}^{k+2}f(k,m,2) =\sum_{k=0}^\infty\sum_{m=1}^{k+1}f(k,m,2)+\sum_{k=0}^\infty f(k,k+2,2) 
\nonumber\\
&=\sum_{m=1}^\infty\sum_{j=0}^{\infty}f(j+m-1,m,2)+\alpha^2\sum_{k=0}^\infty\frac{(-1)^k(2\alpha^2)^k}{k!\Gamma(k+3+\hat\gamma)}
\nonumber \\
&=-\frac{2}{2\alpha}\sum_{m=1}^\infty\frac{(-2\alpha)^k}{(m-1)!\Gamma(m+1+\hat\gamma)}
+\mathscr S(2)+\frac{2^{-\hat\gamma/2}}{2\alpha^{\hat\gamma}}J_{2+\hat\gamma}(2\sqrt 2\alpha) 
\nonumber\\
&=2\frac{2^{(1-\hat\gamma)/2}}{2\alpha^{\hat\gamma}}J_{1+\hat\gamma}(2\sqrt 2\alpha)+
\frac{2^{-\hat\gamma/2}}{2\alpha^{\hat\gamma}}J_{2+\hat\gamma}(2\sqrt 2\alpha)+\mathscr S(2). \label{S2}
\end{align}

\subsection{Evaluating the limiting fraction} Now, we go back to the question of determining
the limit in (\ref{limfrac}). First, the formula for $a_N$ in (\ref{a_N}) equals, by the 
result in (\ref{Li}), 
\begin{align*}
 a_N &\sim B_0R_0L_0 N^{\hat\gamma+1}N!+B_1R_0L_0(N-1)^{\hat\gamma+1}(N-1)! \\
& +B_0R_1L_1(N-1)^{\hat\gamma+2}(N-1)!+B_1R_1L_1(N-2)^{\hat\gamma+2}(N-2)! \\
& +B_0R_2L_2(N-2)^{\hat\gamma+3}(N-2)!+B_1R_2L_2(N-3)^{\hat\gamma+3}(N-3)! \\
&\sim N^{\hat\gamma+1}N!B_0(R_0L_0+R_1L_1+R_2L_2).
\end{align*} 
Recall, from preceding equation (\ref{F_i}), that $B_0=1$ and the $R_i$:s we want to use,
which we have denoted $\sR_i$ and $\tR_i$ for $\sa_N$ and $\ta_N$, respectively, are to 
be found in (\ref{allaR}).

Note also, from  (\ref{IoJ}), that
\begin{align*} 
I(0)+J(0) &=\int_0^1(1-x)^{-\lambda/(2+\lambda)}(x+1)e^{x/(2+\lambda)}\dd x \\
&= \left[-(2+\lambda)(1-x)^{1-\lambda/(2+\lambda)}e^{x/(2+\lambda)} \right]_{x=0}^{x=1}=2+\lambda=1/\alpha,
\end{align*}
and a straightforward property of $\mathscr S(x)$, as defined in (\ref{scrS}), is
\begin{equation}\label{scrSr}
 x\mathscr S(0)+y\mathscr S(1)+z\mathscr S(2)=(x+y+z)\mathscr S\left(\frac{y+2z}{x+y+z}\right).
\end{equation}
Then, from (\ref{allaR}), (\ref{Lif}), (\ref{S0}), (\ref{S1}), (\ref{S2}) and  (\ref{scrS}),
recalling also that $\alpha=1/(2+\lambda)$, we get
\begin{align*}
 &\frac{\sa_N}{N^{\hat\gamma+1}N!} \sim \sum_{i=0}^2 \sR_iL_i=(5+\lambda)L_0-7\alpha L_1+4\alpha^2 L_2 \\
&= \frac{e^{-2\alpha}[I(0)+J(0)]}{\alpha\Gamma(\hat\gamma+1)}+\frac1{\alpha^2}\big[(5+v) S(0)-7 S(1)+4 S(2)\big] \\
&=\frac{1}{\alpha^2}\bigg[\frac{e^{-2\alpha}}{\Gamma(\hat\gamma+1)}+
\frac{2^{(1-\hat\gamma)/2}}{2\alpha^{\hat\gamma}}J_{1+\hat\gamma}(2\sqrt 2\alpha)
+4\frac{2^{-\hat\gamma/2}}{2\alpha^{\hat\gamma}}J_{2+\hat\gamma}(2\sqrt 2\alpha) 
+\frac{1}{\alpha}\mathscr S(\alpha) \bigg],
\end{align*}
where we in the last equality used the result of (\ref{scrSr}). Similarly we get
\begin{align*}
 &\frac{\ta_N}{N^{\hat\gamma+1}N!} \sim \sum_{i=0}^2 \tR_iL_i=(1+\lambda)L_0+\alpha L_1+0 L_2 \\
&=\frac1{\alpha^2}\left[\frac{e^{-2\alpha}}{\Gamma(\hat\gamma+1)}+ 
\frac{2^{(1-\hat\gamma)/2}}{2\alpha^{\hat\gamma}}J_{1+\hat\gamma}(2\sqrt 2\alpha)+
\frac{1}{\alpha}\mathscr S(\alpha) \right].
\end{align*}
Next, we note that 
\begin{align}
 &\mathscr S(\alpha) =\sum_{j=0}^\infty (-2\alpha)^{j}(j+1+\alpha) 
\sum_{m=1}^\infty \frac{(-1)^m(2\alpha^2)^m}{\Gamma(m+1+\hat\gamma)\Gamma(m+1+j)} \nonumber \\
&= \underbrace{\sum_{j=0}^\infty (-2\alpha)^{j}(j+1+\alpha) 
\sum_{m=0}^\infty \frac{(-1)^m(2\alpha^2)^m}{\Gamma(m+1+\hat\gamma)\Gamma(m+1+j)}}_{=:\mathscr S_\lambda}
-\frac{(1-\alpha)e^{-2\alpha}}{\Gamma(\hat\gamma+1)}, \label{scr_}
\end{align}
where we think of the sum denoted $\mathscr S_\lambda$ as a function of the parameter $\lambda$. 

\begin{figure}[hbt]
\includegraphics[scale=0.4]{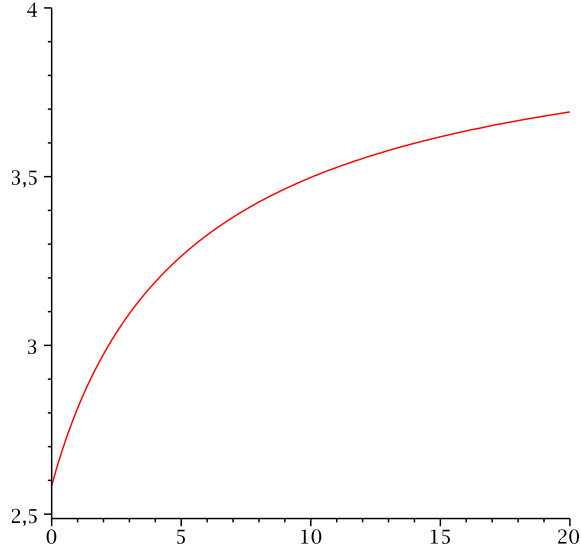}
\caption{The speed of percolation on the ladder with diagonals of (exponential) intensity $\lambda$ on the vertical edges 
(as measured in units of the intensity of the horizontal and diagonal edges), plotted by Maple.}
\label{diagspeed}
\end{figure}

Now, we have an expression for the limiting fraction 
\begin{equation}
 \pi_0=\lim_{n\to\infty}\frac{d_n}{c_n}=\frac
{ \frac{(2\alpha-1)e^{-2\alpha}}{\Gamma(\hat\gamma+1)}+
\frac 12(\sqrt 2\alpha)^{1-\hat\gamma}J_{1+\hat\gamma}
+\mathscr S_\lambda  }
{ \frac{(2\alpha-1)e^{-2\alpha}}{\Gamma(\hat\gamma+1)}+
\frac 12(\sqrt 2\alpha)^{1-\hat\gamma}\big[ J_{1+\hat\gamma}
+2\sqrt 2 J_{2+\hat\gamma}\big]
+\mathscr S_\lambda }, \label{lf}
\end{equation}
in which we have suppressed the argument $2\sqrt2\alpha$ from the Bessel function, i.e.\
in  (\ref{lf}) we should interpret
$J_\nu$ as $J_\nu(2\sqrt2\alpha)$. The speed 
\[ \speed=2+2\pi_0\]
is plotted in Figure \ref{diagspeed} for 
$\lambda$ ranging from 0 to 20. At zero the speed is 
$\frac{2\sqrt 2J_1(\sqrt 2)}{3J_1(\sqrt2)/\sqrt2-J_0(\sqrt 2)}\approx 2.58$ - this exact expression
is derived in  Section \ref{eq0} below - and as $\lambda$ increases to infinity, the speed 
must reach 4. 

In Figure \ref{exempel2} we have given numerical values when $\lambda$ is 0, 1 and 2, respectively.

\begin{figure}[hbt]
 \subfigure[]{\includegraphics[scale=0.40]{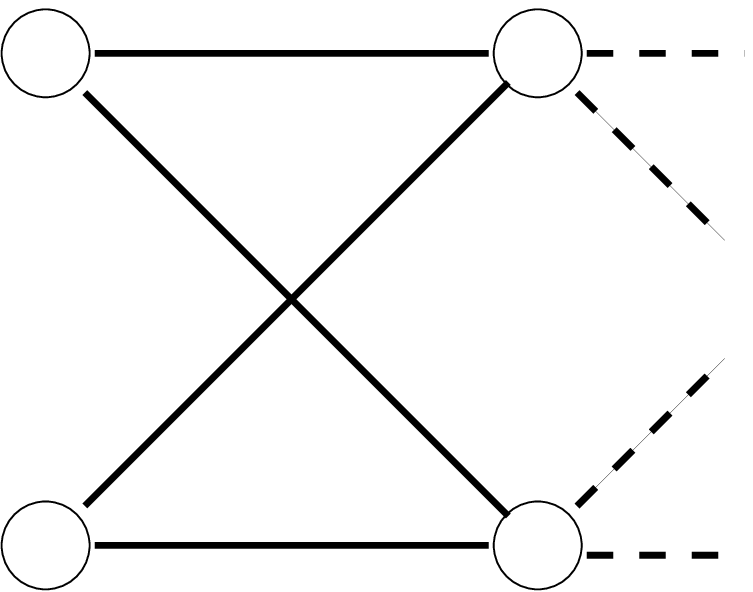}}
 \subfigure[]{\includegraphics[scale=0.40]{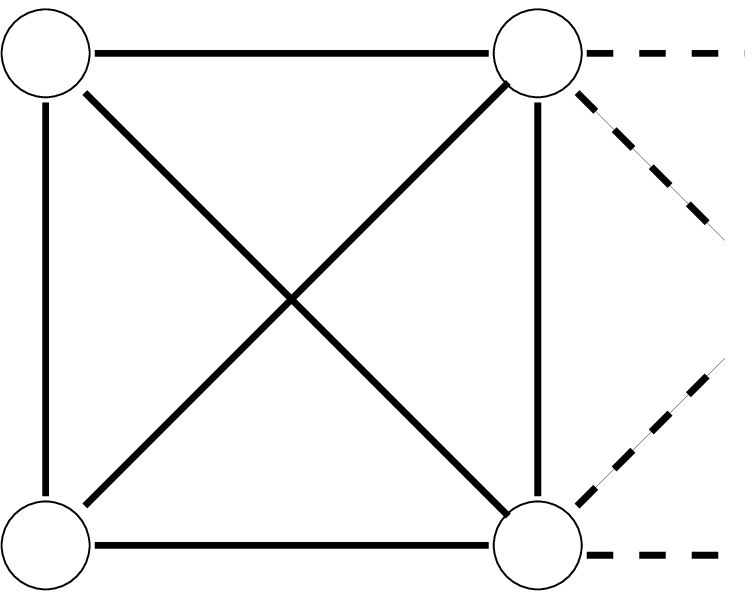}}
 \subfigure[]{\includegraphics[scale=0.40]{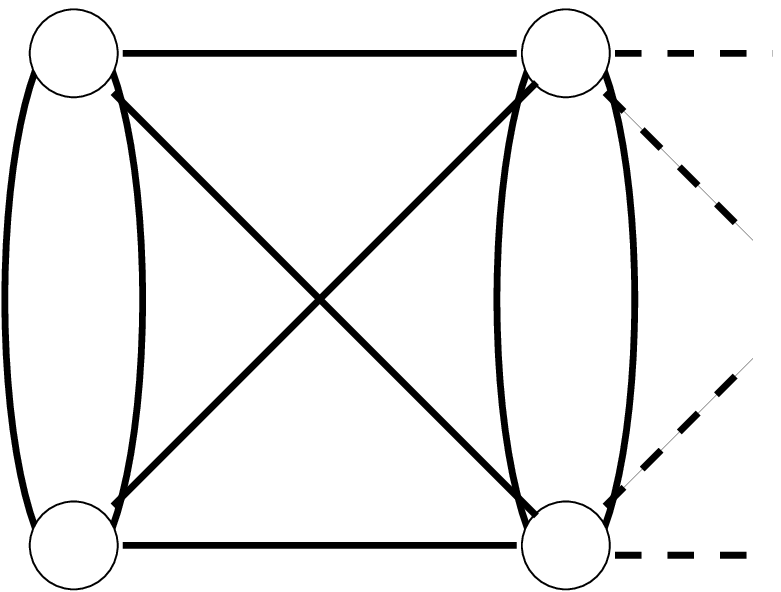}}
\caption{The speed of percolation when all edges are associated with random variables having intensity 1. Numerical
values to two digits are, from left to right, 2.58, 2.81 and 2.97. }
\label{exempel2}
\end{figure}

\subsection{The case $\lambda=0$} \label{eq0}
This special case has already been investigated in \cite{Sch09} by other methods. In this section 
we show that our method gives the same result, since this is far from obvious. 
Also, the case of $\lambda=0$ appears to be
the only case that allows for great simplification of the expression (\ref{lf}).
We note that
for $\lambda=0$, since this implies $\alpha=1/2$ and $\hat\gamma=0$, we get from the definition
of $\mathscr S_\lambda$ in (\ref{scr_}),
\begin{align}
 \mathscr S_0 &=
\sum_{j=0}^\infty (-1)^{j}\left(j+\frac32\right) 
\sum_{m=0}^\infty \frac{(-1)^m(1/2)^m}{m!(m+j)!} \nonumber \\
&=\sum_{j=0}^\infty (-1)^{j}\left(j+\frac32\right)(\sqrt 2)^jJ_j(\sqrt2) \nonumber \\
&=\frac 32\sum_{j=0}^\infty (-1)^{j}(\sqrt 2)^jJ_j(\sqrt2) +
\sum_{j=0}^\infty (-1)^{j}(\sqrt 2)^{j-1} \sqrt2jJ_j(\sqrt2) \label{S_0} \\
\intertext{and since $\sqrt2jJ_j(\sqrt2)=J_{j-1}(\sqrt2)+J_{j+1}(\sqrt2)$ by relation
(\ref{Brec}), we get, after cancellation of terms} \nonumber 
\mathscr S_0 &= \frac 12[J_0(\sqrt 2)-\sqrt2J_1(\sqrt 2)]. 
\end{align}
This, together with the fact that $J_2(\sqrt 2)$ that now appears in the 
denominator of (\ref{lf}), as a consequence of relation (\ref{Brec}), can be written as 
$\sqrt 2 J_1(\sqrt 2)-J_0(\sqrt 2)$, simplifies (\ref{lf}) to
\[ \frac{J_0(\sqrt 2)-J_1(\sqrt 2)/\sqrt 2}{-J_0(\sqrt 2)+3J_1(\sqrt 2)/\sqrt 2}, \]
so that 
\begin{align*}
 \speed &=2+2\pi_0=\frac{2\sqrt 2J_1(\sqrt 2)}{3J_1(\sqrt2)/\sqrt2-J_0(\sqrt 2)},
\end{align*}
and thus
\[ \rate=\frac 1\speed=\frac 34-\frac{J_0(\sqrt 2)}{2\sqrt2 J_1(\sqrt 2)} \]
which is what is calculated in \cite{Sch09}.

\section{Other graphs in $\LL$}
The class $\LL$ of graphs, defined in Section \ref{section2} and Figure \ref{LL}, contains 
(at least) 3 more graphs of interest; these are shown in Figure \ref{andra}. We will
not deal with these graphs here but make some remarks. We believe that the front process
together with the results  of \cite{Jan10}, used in Section \ref{diag}, can be used to calculate the speed
of percolation on the graph depicted in Figure \ref{c}, in principle with arbitrary
intensities associated with the vertical, diagonal and the two different horizontal
edges. There would be added complexity as the state space of the front process
would be all the integers, not only the non-negative ones. The speed when the intensities
are the same is $\frac{2\tan 1-1}{2\tan 1-2}\approx 1.90$, from the exact expression
of the rate of percolation calculated in \cite{Sch09}.
 
The problem when trying to apply the front process to the other graphs is that the Markov property
is lost. 

\begin{figure}[hbt]
 \subfigure[]{\includegraphics[scale=0.45]{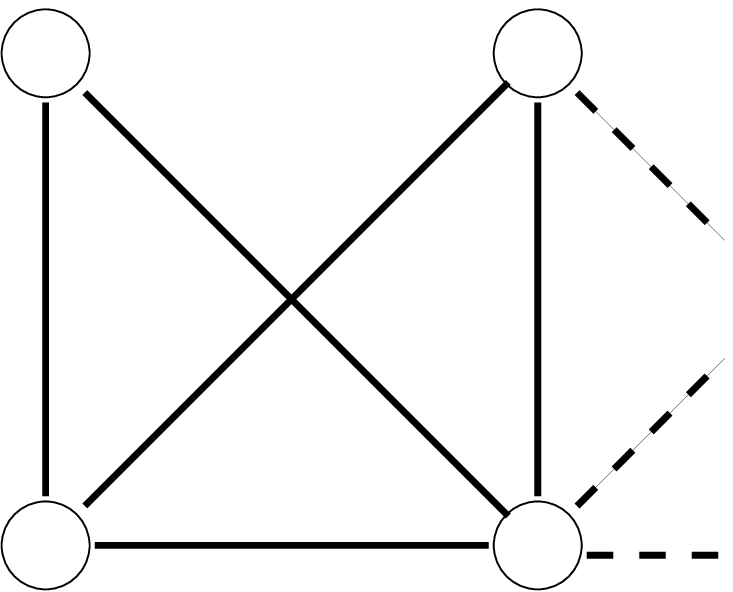}}
 \subfigure[]{\includegraphics[scale=0.45]{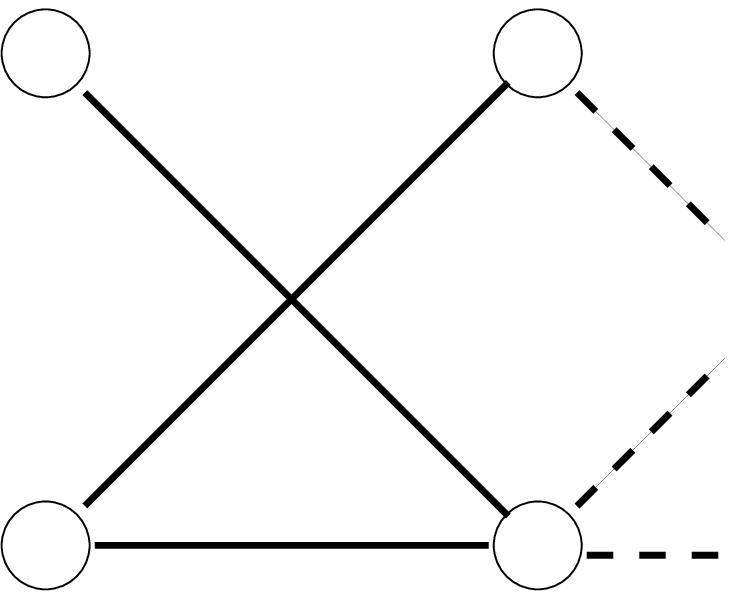}}
 \subfigure[]{\includegraphics[scale=0.45]{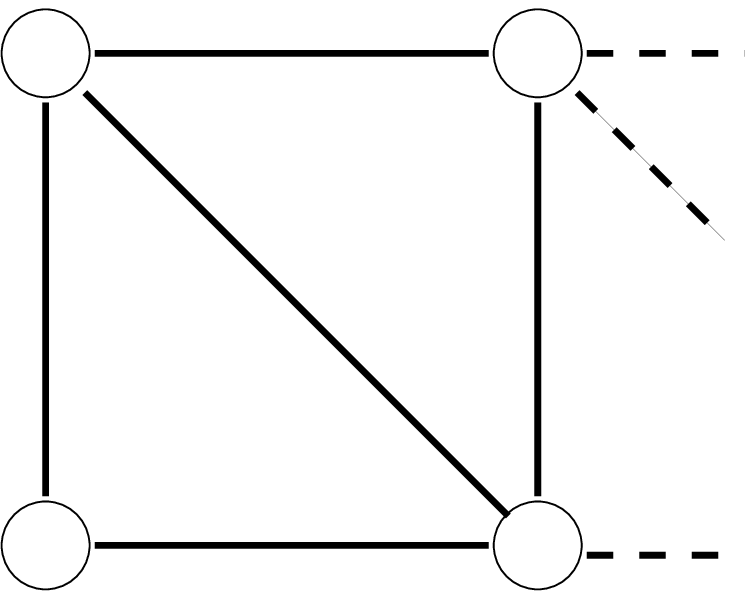} \label{c} }  
\caption{Three more graphs in $\LL$.  }
\label{andra}
\end{figure}

\appendix

\section{Proof of Claim \ref{theLi}} \label{pLi}
\noindent \textbf{Step 1.}
Let $\alpha$ and $K$ be nonnegative real numbers, $\alpha$ is considered fixed,
and define 
\begin{equation} \label{Adef}
\ett_m:=\ett_m(K)=\sum_{j=1}^m \binom mj\alpha^j(m+K-j)!. 
\end{equation}
Then we can write $\ett_m$ as
\[ \ett_m=(m+K)!\sum_{j=0}^m\frac{\alpha^j}{j!}\prod_{i=0}^{j-1}\frac{m-i}{m+K-i}. \]
We are interested in the limit of $\ett_m'=\ett_m/(m+K)!$ as $m$ tends to infinity.

Now, as $(m-i)/(m+K-i)$ is decreasing in $i$, 
\[ \prod_{i=0}^{j-1}\frac{m-i}{m+K-i}\leq \left(\frac{m}{m+K}\right)^j, \]
and as a consequence, we get 
\begin{equation}
\label{Aupp} \ett_m'\leq e^{\alpha m/(m+K))}.
 \end{equation} 

Next, as long as $j<m$,  since $\ln[(m-i)/(m+K-i)]$ is a negative and decreasing function
when $0\leq i<m$, we can calculate,
\begin{align}
 \prod_{i=0}^{j-1}&\frac{m-i}{m+K-i} = 
\exp\left\{\sum_{i=0}^{j-1}\ln\left[\frac{m-i}{m+K-i} \right] \right\} 
\geq \exp\left\{\int_0^j \ln\left[\frac{m-x}{m+K-x}\right] \right\} \nonumber \\
&=\frac{m^m}{(m+K)^{m+K}}\frac{(m+K-j)^{m+K-j}}{(m-j)^{m-i}}= \nonumber \\
&=\left(\frac{m}{m+K}\right)^j
\underbrace{\frac{ \left(1-\frac{j}{m+K}\right)^{m+K} }  {\left(1-\frac{j}{m}\right)^{m}}}_{\geq 1}
\left(\frac{ 1-\frac{j}{m} }{1-\frac{j}{m+K}}\right)^j. \label{3fac}
\end{align}
Since
\[ \frac{ 1-\frac{j}{m} }{1-\frac{j}{m+K}}=1-\frac{jK}{m(m+K-i)},\]
the last factor in (\ref{3fac}) is, by Taylor expansion around 1, 
\[ 1-\frac{j^2K}{m(m+K-j)}+\bigo\left(\left[\frac{K}{m(m+K-j)}\right]^2\right), \]
and hence  $\ett_m'$ is bounded below by 
\begin{equation}\label{Aned}
 \sum_{j=1}^{m-1}\frac{\left[m\alpha/(m+K)\right]^j}{j!}(1+\bigo(1/m))=e^{\alpha m/(m+K)}(1+\bigo(1/m)).
\end{equation}
In conclusion; from the upper and lower bounds, (\ref{Aupp}) and (\ref{Aned}) respectively,
 \begin{equation} \label{Aet}
  \ett_m=(m+K)!\exp\left\{\frac{\alpha m}{m+K}\right\}(1+\bigo(1/m)).
 \end{equation}

\noindent \textbf{Step 2.}
Let us examine, for $l\geq 0$ and $K>-1$, the fraction $\Gamma(l+K+1)/\Gamma(l+1)$ for
large $l$. By Stirling's formula
\begin{align}
 \frac{\Gamma(l+K+1)}{\Gamma(l+1)} &=
\frac{\sqrt{2\pi}(l+K)^{l+K+1/2}e^{-l-K}}{\sqrt{2\pi}l^{l+1/2}e^{-l}}(1+\bigo([l+K]^{-1})) \nonumber \\
&=(l+K)^{l+K}e^{-K}\left(1+\frac{K}{l}\right)^{l+1/2}(1+\bigo([l+K]^{-1})) \nonumber \\
&=(l+K)^{l+K}(1+\bigo(1/l)) \label{Gkvot}
\end{align}

\noindent \textbf{Step 3.}
Define the function
\begin{equation*}
 \tva_i(k)= 
\sum_{l+m=M-2k}\frac{\Gamma(k+i+\hat\gamma+l+1)}{\Gamma(l+1)}(m+1)\ett_m(k+l), \;\;k\leq M/2,
\end{equation*}
where $\ett_m(k+l)$ is defined in  (\ref{Adef}). 
With $\tva_i(k)$  defined in this way we can
write $\hat F_i(M)$ from (\ref{hatF})  as
\[ \hat F_i(M)=\sum_{k\leq M/2} \frac{(-C)^k}{k!\Gamma(k+i+\hat\gamma+1)} \tva_i(k). \]
Now, examine $\tva_i(k)$ scaled by $M^{k+\hat\gamma+i+1}(M-k)!$,
\begin{align*}
 \frac{\tva_i(k)}{M^{k+\hat\gamma+i+1}(M-k)!}&=
\sum_{l+m=M-2k}
\underbrace{\frac{\Gamma(k+i+\hat\gamma+l+1)}{\Gamma(l+1)M^{k+\hat\gamma+i}}}_{a}\cdot
\underbrace{\frac{m+1}{M}}_{b}\cdot
\underbrace{\frac{\ett_m(k+l)}{(M-k)!}}_{c}.
\end{align*}
From (\ref{Gkvot}) we get, using $l+m=M-2k$,
\begin{align*}
 a&= \left(\frac{\hat\gamma+i-k}{M}+ 1-\frac mM \right)^{k+i+\hat\gamma}(1+\bigo(1/l)) \\
&=\left( 1-\frac mM \right)^{k+i+\hat\gamma}(1+\bigo(1/l)).
\end{align*}
Obviously $b=m/M+1/M$. From (\ref{Aet}) we get 
\begin{align*}
 c&=\exp\left\{\frac{\alpha m}{M-k}\right\}(1+\bigo(1/m))=
\exp\left\{\frac{\alpha m}{M}\right\}(1+\bigo(1/m)).
\end{align*}
So that, as $M\to\infty$,
\begin{align*} 
 \frac{\tva_i(k)}{M^{k+\hat\gamma+i+1}(M-k)!}& \sim
 \sum_{m=0}^{M} \left(1-\frac mM\right)^{k+i+\hat\gamma}\frac mM\exp\left\{\alpha \frac mM\right\} \\
&\sim \int_0^1(1-x)^{k+i+\hat\gamma}xe^{\alpha x} \dd x.
\end{align*}
And hence, since also $M^{k+\hat\gamma+i+1}(M-k)!\sim M^{\hat\gamma+i+1}M!$, we have shown
the result in (\ref{Li}).

\section{Proof of Claim \ref{theIn} } \label{pIn}
Recall that we have defined, for  $n\in\mathbb N$,
\begin{align} 
 I(n) =\int_0^1(1-x)^{\hat\gamma+n}xe^{\alpha x}\dd x \quad\mbox{and} \quad
 J(n) =\int_0^1(1-x)^{\hat\gamma+n}e^{\alpha x}\dd x,
\end{align}
where $\hat\gamma\in(-1,0]$ and $\alpha\in(0,1/2]$, although 
in this section it is only essential that $\hat\gamma>-1$.

Then $I(0)<\infty$ and $J(0)<\infty$, since $\hat\gamma>-1$. For $n\geq 1$,
partial integration reveals the recursive relationship
\begin{align*}
 I(n)&=\left[(1-x)^{\hat\gamma+n}x\frac1{\alpha}e^{\alpha x}\right]_{x=0}^{x=1} \\ 
&\phantom{=} -\frac 1\alpha\int_0^1[-(\hat\gamma+n)(1-x)^{\hat\gamma+n-1}x+(1-x)^{\hat\gamma+n}]e^{\alpha x}\dd x \\
&=\frac 1\alpha I(n-1)-\frac1{\alpha}J(n),
\end{align*}
from which follows by iteration that
\begin{equation}
 I(n)= \frac{(\hat\gamma+n)^{\ff n}}{\alpha^n}I(0)
-\frac1\alpha\sum_{j=1}^{n}\frac{(\hat\gamma+n)^{\ff{n-j}}}{\alpha^{n-j}}J(j). \label{Irec}
\end{equation}
Note that the above formula holds for any $n\geq 0$ if we interpret sums of the
form $\sum_n^{m}$ to be zero if $m<n$. 

Next,  we turn our attention to $J(n)$. For $n\geq 1$ we get, again by partial integration,
\begin{align*}
 J(n) &=\left[(1-x)^{\hat\gamma+n}\frac1{\alpha}e^{\alpha x}\right]_{x=0}^{x=1}
+\frac{\hat\gamma+n}\alpha\int_0^1[(1-x)^{\hat\gamma+n-1}e^{\alpha x}\dd x \\
&=\frac{\hat\gamma+n}{\alpha}J(n-1)-\frac1\alpha,
\end{align*}
from which follows by iteration that
\begin{equation}\label{Jform}
 J(n)= \frac{(\hat\gamma+n)^{\ff n}}{\alpha^n}J(0)-\frac 1\alpha\sum_{k=0}^{n-1}\frac{(\hat\gamma+n)^{\ff k}}{\alpha^k}.
\end{equation}
Using (\ref{Jform}) in (\ref{Irec}) yields
\begin{align*}
 I(n)&=\frac{(\hat\gamma+n)^{\ff n}}{\alpha^{n}}I(0)-\frac{J(0)}{\alpha^{n+1}}S_1
+ \frac{1}{\alpha^2}S_2, 
\end{align*}
where
\begin{align*}
S_1 &= \sum_{j=1}^{n}(\hat\gamma+n)^{\ff {n-j}}(\hat\gamma+j)^{\ff{j}}=\frac{n\Gamma(\hat\gamma+n+1)}{\Gamma(\hat\gamma+1)}, \quad\mbox{and} \\
S_2 &= \sum_{j=1}^{n}\sum_{k=0}^{j-1}\frac{(\hat\gamma+n)^{\ff {n-j}}(\hat\gamma+j)^{\ff{k}}}{\alpha^{k-j}}.
\end{align*}
The above simplification of $S_1$ is easily seen to be true 
by the relation $x^{\ff n}=\Gamma(x+1)/\Gamma(x+1-m)$. The same relation can be applied to
$S_2$ to give
\begin{align*}
 S_2&=\frac{\Gamma(\hat\gamma+n+1)}{\alpha^n}\sum_{j=1}^{n}\sum_{k=0}^{j-1}\frac{\alpha^{j-k}}{\Gamma(\hat\gamma+1+j-k)} \\
&=\frac{\Gamma(\hat\gamma+n+1)}{\alpha^n}\sum_{m=1}^{n}\frac{(n+1-m)\alpha^{m}}{\Gamma(\hat\gamma+1+m)},
\end{align*}
where the last equality follows from noting that $k$ and $j$ only appear in the summation as $m=j-k$, and 
that the values of $m$, i.e.\ 1, 2, $\ldots$, $n$, appear $n$, $n-1$, $\ldots$, $1$ times, respectively. 

So, we have the sought formula
\begin{align}
 I(n) &=\frac{\Gamma(\hat\gamma+n+1)}{\alpha^n\Gamma(\hat\gamma+1)}\bigg[I(0)-\frac{n}{\alpha}J(0)\bigg] \nonumber \\
&\phantom{=} +\frac{\Gamma(\hat\gamma+n+1)}{\alpha^{n+2}}\sum_{m=1}^{n}\frac{(n+1-m)\alpha^{m}}{\Gamma(\hat\gamma+1+m)} \nonumber \\
&=\frac{\Gamma(\hat\gamma+n+1)}{\alpha^n}\bigg[ \frac{I(0)-\frac{n}{\alpha}J(0)}{\Gamma(\hat\gamma+1)}+
\frac1{\alpha^2}\sum_{m=1}^{n}\frac{(n+1-m)\alpha^{m}}{\Gamma(\hat\gamma+1+m)}\bigg]. \label{Jf}
\end{align}


\begin{thebibliography}{FGS06}
\bibitem[AS64]{H64} M.\ Abramowitz, I.\ Stegun: \emph{Handbook of mathematical functions}. National Bureau of Standards
Applied Mathematics Series 55 (1964).

\bibitem[FGS06]{FGS06} 
A.\ Flaxman, D.\ Gamarnik, G.\ Sorkin: 
First-passage percolation on a width-2 strip and the path cost in a VCG auction. 
\emph{Internet and Network Economics},
\textbf{4286} (2006) 99-111.

\bibitem[Jan10]{Jan10}
S.\ Janson:
A divergent generating function that can be summed and analysed analytically
\emph{Disc.\ Math.\ Theoret.\ Comput.\ Sci.},  
\textbf{12} no.\ 2 (2010)   1--22.

\bibitem[Ren10]{Ren10}
H.\ Renlund: 
First-passage percolation with exponential times on a ladder.
\emph{Combin. Probab. Comput},
\textbf{19} no.\ 4 (2010) 593--601.

\bibitem[Sch09]{Sch09} 
E.\ Schlemm: 
First-passage percolation rates on width-two stretches with exponential link weights.
\emph{Electron.\ Commun.\ Probab.}, 
\textbf{14} (2009) 424--434.



\bibitem[SW78]{SW78} 
R. T.\ Smythe, J. C.\ Wierman: 
\emph{First-Passage Percolation on the square lattice}. 
Lecture notes in mathematics \textbf{671}, Springer (1978).

\bibitem[Wat22]{Wat22} 
G.\ N.\ Watson: 
\emph{A Treatise on the Theory of Bessel Functions}.
Cambridge University Press (1922).
\end{thebibliography}
\end{document}